\documentclass[a4paper,11pt]{article}
\usepackage[utf8x]{inputenc}
\usepackage[left=20mm, right=20mm, top=20mm, bottom=20mm]{geometry}
\usepackage{color, hyperref}
\hypersetup{colorlinks=true, linkcolor=blue, citecolor=red, urlcolor=magenta} 
\usepackage{amsmath}
\usepackage{amssymb}
\usepackage{amsfonts}
\usepackage{amsthm,amscd} 
\usepackage{authblk}
\usepackage{tikz-cd} 

\newtheorem{theorem}{Theorem}[section] 
\newtheorem{lemma}[theorem]{Lemma}     
\newtheorem{corollary}[theorem]{Corollary}
\newtheorem{proposition}[theorem]{Proposition}
\theoremstyle{definition} 
\newtheorem{definition}[theorem]{Definition}
\newtheorem{example}[theorem]{Example}

\theoremstyle{remark} 
\newtheorem{remark}[theorem]{Remark}
\newtheorem{thmx}{{\bf Theorem}} % For main theorems in Introduction (lettered)

\newtheorem{propx}[thmx]{{\bf Proposition}} % For main propositions in Introduction (lettered)
 
\numberwithin{equation}{section} 

\renewcommand{\imath}{\iota} 
\DeclareMathOperator{\id}{id} 
\DeclareMathOperator{\pr}{pr} 
 
\DeclareMathOperator{\ev}{ev} 
\DeclareMathOperator{\Lie}{L} 

\DeclareMathOperator{\imathOp}{\iota} % Operator for interior product
\DeclareMathOperator{\Tr}{Tr} % Trace operator

\title{Geometric structures on Weil bundles: Canonical differential-geometric constructions} 

\author[1]{S. Tchuiaga\thanks{tchuiagas@gmail.com}}
\author[2]{A. Ndiaye\thanks{ameth1.ndiaye@ucad.edu.sn}}
\author[3]{C. Khoule\thanks{cheikh1.khoule@ucad.edu.sn}}
\author[4]{R. A. M. Mohameden\thanks{ranyaahmed3611@gmail.com}}

\affil[1]{Department of Mathematics, University of Buea, South West Region, Cameroon}
\affil[2, 3]{D\'{e}partement de Math\'{e}matiques, Universit\'{e} Cheikh Anta Diop, Dakar, Senegal}
\affil[4]{Department of Mathematics, Omdurman Islamic University, Sudan}

\providecommand{\keywords}[1]{\textbf{\textit{Keywords:}} #1}
\date{ } 

\begin{document}
	\maketitle
	
	\begin{abstract}
		This paper investigates the transfer of classical geometric structures from a smooth manifold $M$ to its Weil bundle $(M^\mathbf A, \tilde\pi_M, M)$ associated with a Weil algebra $\mathbf A$. We show that various structures—including locally conformal symplectic (lcs), locally conformal cosymplectic (lcc), contact, Jacobi, Sasakian, Walker, sub-Riemannian, orientation, Riemannian, and K\"ahlerian structures—admit canonical lifts to $M^\mathbf A$. Our approach emphasizes the differential-geometric properties of these canonical constructions, utilizing the Weil projection $\tilde{\pi}_M$ and related functorial tools. This provides a unified perspective on endowing Weil bundles with rich geometric structure inherited from the base manifold. Furthermore, we highlight a specific construction yielding a cosymplectic manifold on $M^\mathbf{A}$ (for suitable $M$ and $\mathbf{A}$) that is demonstrably not a trivial suspension of a symplectic manifold. We also explicitly show how integrability of almost complex structures is preserved and clarify the nature of lifted characteristic vector fields.
	\end{abstract}	
	
	\keywords{Weil bundle, canonical lift, Weil prolongation, Jacobi structure, symplectic structure, cosymplectic structure, contact structure, Kahler structure, Sasakian structure, Walker manifold, Sub-Riemannian structure}
	
	\vspace{0.5cm} 
	\textbf{2020 Mathematics Subject Classification:} 58A32, 53Dxx, 53Cxx, 53Bxx. 
	\vspace{0.5cm} 
	
	\section{Introduction}\label{sec1}
	The theory of Weil bundles $(M^\mathbf A, \tilde\pi_M, M)$, originating from André Weil's concept of infinitely near points \cite{Wei}, provides a powerful framework for studying differential geometry by functorially associating a new manifold $M^\mathbf A$ to a given manifold $M$ and a Weil algebra $\mathbf A$. This construction, $T^\mathbf{A}: M \mapsto M^\mathbf{A}$, known as the Weil functor, naturally extends maps $f:M \to N$ to $f^\mathbf{A}: M^\mathbf{A} \to N^\mathbf{A}$ and finds applications in areas like jet bundles, synthetic differential geometry, and the prolongation of structures \cite{K-1, Mo, Ok, G-K, KMS}.
	
	Recent studies have explored specific geometric structures on $M^\mathbf A$, notably demonstrating that if $M$ is symplectic, then $M^\mathbf A$ admits a naturally induced symplectic structure \cite{A-F-J}. The dimension relation $\dim(M^\mathbf A) = \dim(M) \cdot \dim(\mathbf A)$ plays a crucial role, sometimes requiring conditions on $\dim(\mathbf A)$ (e.g., being odd) for structures like contact or cosymplectic geometry.
	
	While the simple pullback operation $\tilde{\pi}_M^*$ associated with the Weil projection provides a map from forms on $M$ to forms on $M^\mathbf A$, it often yields degenerate structures when applied to metrics, symplectic forms, or contact forms (unless $\dim \mathbf{A}=1$). However, it is well-established in the literature that canonical lifts or prolongations of these geometric structures to $M^\mathbf A$ exist, leveraging the full algebraic structure of $\mathbf A$ and the functorial nature of $T^\mathbf A$ \cite{KMS, Mo}. These canonical lifts produce non-degenerate structures of the same type (symplectic, Riemannian, contact, etc.) on $M^\mathbf A$. Throughout this paper, symbols like $\omega^\mathbf{A}, g^\mathbf{A}, \beta^\mathbf{A}$ will denote these canonical lifts, unless otherwise specified.
	
	This paper aims to provide a concise overview of how various classical geometric structures on $M$ canonically induce corresponding structures on $M^\mathbf A$. We adopt a differential-geometric perspective, focusing on the properties of these canonical lifts and their relationship with the base structures via the projection $\tilde{\pi}_M$. Our main results demonstrate the preservation of structure type under these canonical Weil prolongations for lcs, lcc, cosymplectic, contact, Jacobi, Sasakian, Walker, sub-Riemannian, orientation, Riemannian, and K\"ahler structures. We explicitly construct the Reeb and Killing fields on $M^\mathbf{A}$ in terms of their base counterparts using sections. Furthermore, we provide an explicit example confirming that Weil bundles can yield cosymplectic manifolds which are not simple suspensions $P \times \mathbb{R}$. We also include a direct proof of the preservation of integrability for almost complex structures under the canonical lift.\\ 
	
	While the simple pullback operation $\tilde{\pi}_M^*$, associated with the Weil projection $\tilde{\pi}_M: M^\mathbf{A} \to M$, provides a map from forms on $M$ to forms on $M^\mathbf A$, it often yields degenerate structures when applied to metrics, symplectic forms, or contact forms (unless $\dim \mathbf{A}=1$). For instance, if $\omega = \sum dx_i \wedge dy_i$ locally on $M$, then $\tilde{\pi}_M^* \omega = \sum dx_{i,1} \wedge dy_{i,1}$ in the natural coordinates $(x^{i,k})$ on $M^\mathbf{A}$ (where $x^{i,1} = x^i \circ \tilde{\pi}_M$), which is clearly degenerate for $l>1$. However, it is well-established in the literature that canonical lifts or prolongations of these geometric structures to $M^\mathbf A$ exist, lifts the full algebraic structure of $\mathbf A$ and the functorial nature of $T^\mathbf A$ \cite{KMS, Mo}. These canonical lifts produce non-degenerate structures of the same type (symplectic, Riemannian, contact, etc.) on $M^\mathbf A$. Throughout this paper, symbols like $\omega^\mathbf{A}, g^\mathbf{A}, \beta^\mathbf{A}$ will denote these non-degenerate canonical lifts, explicitly distinct from the degenerate pullbacks $\tilde{\pi}_M^* \omega, \tilde{\pi}_M^* g, \tilde{\pi}_M^* \beta$, unless otherwise specified.\\
	
	Our main results are summarized as follows (refer to Appendix \ref{app:definitions} for definitions):
	
	\begin{thmx}\label{Theo-LCS-LCC-Cosymp} 
		Let $\mathbf A$ be a Weil algebra of dimension $l$.
		\begin{enumerate}
			\item If $(M, \omega, \theta)$ is a locally conformal symplectic (lcs) manifold, then $M^\mathbf A$ admits a canonical lcs structure $(\omega^\mathbf A, \theta^\mathbf A)$.
			\item If $(M, \omega, \eta, \theta)$ is a locally conformal cosymplectic (lcc) manifold and $l$ is odd, then $M^\mathbf A$ admits a canonical lcc structure $(\omega^\mathbf A, \eta^\mathbf A, \theta^\mathbf A)$.
			\item If $(M, \omega, \eta)$ is a cosymplectic manifold with Reeb field $\xi_{M}$ and $l$ is odd, then $M^\mathbf A$ admits a canonical cosymplectic structure $(\omega^\mathbf A, \eta^\mathbf A)$. Its Reeb vector field $\xi_{M^\mathbf A}$ satisfies $(\tilde\pi_M)_\ast(\xi_{M^\mathbf A}) = \xi_M$ and can be constructed as $\xi_{M^\mathbf A} = \frac{1}{l}\sum_{j = 1}^l(S_j)_\ast(\xi_M)$ using suitable sections $S_j$ associated with the definition of $(\omega^\mathbf{A}, \eta^\mathbf{A})$.
		\end{enumerate}
	\end{thmx}
	
	\begin{thmx}\label{Theo-Riemannian}
		Let $(M, g)$ be a Riemannian manifold, and $\mathbf A$ be a Weil algebra. Then $(M^\mathbf A, g^\mathbf A)$ is a Riemannian manifold, where $g^\mathbf{A}$ is the canonical lift of $g$. Furthermore:
		\begin{enumerate}
			\item The canonical connection $\nabla^\mathbf A$ is the Levi-Civita connection for $g^\mathbf A$.
			\item Geodesics lift to geodesics via canonical sections.
			\item The canonical lift $X^\mathbf{A}$ of a Killing field $X$ on $M$ is a Killing field for $g^\mathbf{A}$. Conversely, the projection of a Killing field on $M^\mathbf{A}$ is Killing on $M$.
		\end{enumerate}
	\end{thmx}
	
	\begin{thmx}\label{Theo-Kahler}
		Let $(M, g, \omega, J)$ be a K\"ahler manifold, and $\mathbf A$ be a Weil algebra. Then, $(M^\mathbf A, g^\mathbf A, \omega^\mathbf A, J^\mathbf A)$ is a canonical K\"ahler structure. In particular, $J^\mathbf{A}$ is integrable.
	\end{thmx}
	
	\begin{thmx}\label{Theo-Contact} 
		Let $(M, \beta)$ be a contact manifold with Reeb field $\xi$, and let $\mathbf A$ be a Weil algebra of odd dimension $l$. 
		Then $(M^\mathbf A, \beta^\mathbf A)$ is a canonical contact structure, where $\beta^\mathbf{A}$ is the lift of $\beta$. The Reeb vector field is $\xi_{M^\mathbf A} = \xi^\mathbf A$.
	\end{thmx}
	
	\begin{thmx}\label{Theo-Orientation}
		If $M$ is an oriented manifold, its Weil bundle $M^\mathbf A$ is canonically orientable.
	\end{thmx}
	
	\begin{thmx}\label{Theo-Jacobi}
		If $(M, \Lambda, \Xi)$ is a Jacobi manifold, then $M^\mathbf A$ admits a canonical Jacobi structure $(\Lambda^\mathbf A, \Xi^\mathbf A)$ defined via averaging lifts through sections.
	\end{thmx}
	
	\begin{thmx}\label{Theo-Sasakian} 
		If $(M, g, \eta, \xi, \Phi)$ is a Sasakian manifold and $\dim \mathbf A$ is odd, then $(M^\mathbf{A}, g^\mathbf{A}, \eta^\mathbf{A}, \xi^\mathbf{A}, \Phi^\mathbf{A})$ is a canonical Sasakian structure.
	\end{thmx}
	
	\begin{propx}\label{Prop-SubRiemannian} 
		If $(M, \mathcal{D}, g_{\mathcal{D}})$ is a bracket-generating sub-Riemannian manifold, then $(M^\mathbf{A}, \mathcal{D}^\mathbf{A}, (g_{\mathcal{D}})^\mathbf{A})$ inherits a canonical bracket-generating sub-Riemannian structure.
	\end{propx}
	
	\begin{propx}\label{Prop-Walker}
		If $(M, g)$ is a Walker manifold with parallel null distribution $\mathcal{D}$, then $(M^\mathbf{A}, g^\mathbf{A})$ is a Walker manifold with canonical parallel null distribution $\mathcal{D}^\mathbf{A}$.
	\end{propx}
	
	The paper is organized as follows. Section \ref{sec2} reviews basic notions of Weil algebras, Weil bundles, local coordinates, and introduces the crucial concept of canonical lifts via the Weil functor, including comparisons between different lift types. Section \ref{sec3} contains the proofs of the main theorems and propositions. Section \ref{sec:remarks} includes remarks on structures not preserved (Einstein, Calabi-Yau). Section \ref{sec:examples} provides explicit examples of the lifted structures on Euclidean spaces. Section \ref{sec:conclusion} offers concluding thoughts and future research directions. An appendix recalls the definitions of the geometric structures involved.
	
	\section{Basic notions}\label{sec2} 
	
	\begin{definition}[\cite{Wei}]\label{def:WeilAlgebra}
		A \textbf{Weil algebra} is a commutative, associative, unitary $\mathbb R$-algebra $\mathbf A$ of finite dimension $l$, which has a unique maximal ideal $\mathcal{A}$ such that $\mathbf A / \mathcal{A} \cong \mathbb R$ and $\mathcal{A}^k = \{0\}$ for some integer $k \ge 1$. Equivalently, $\mathbf A = \mathbb R \cdot 1_\mathbf{A} \oplus \mathcal{A}$ (vector space direct sum).
	\end{definition}
	We denote the projection onto the real part by $\pr_{\mathbb R}: \mathbf A \to \mathbb R$. An element $x \in \mathbf A$ decomposes as $x = x_{\mathbb R} \cdot 1_\mathbf{A} + x_\circ$, with $x_{\mathbb R} = \pr_{\mathbb R}(x) \in \mathbb R$, and $x_\circ \in \mathcal{A}$. We often assume a canonical trace map $\Tr: \mathbf{A} \to \mathbb{R}$, usually a projection onto the $\mathbb{R} \cdot 1_\mathbf{A}$ component, normalized by $\Tr(1_\mathbf{A})=1$.
	
	% Example and Definition of M^A remain the same...
	\begin{example}[\cite{K-1, Mo}] Let $\mathbb R [X_1, \dots, X_n]$ be the algebra of polynomials in $n$ indeterminates.
		\begin{enumerate}
			\item The algebra of dual numbers: $T\mathbb R := \mathbb R [X] /(X^2) \cong \mathbb R \oplus u \mathbb R$, with $u^2 = 0$. Here $\mathcal{A} = u\mathbb{R}$, $l=2$.
			\item The algebra of $k$-jets of functions $\mathbb{R} \to \mathbb{R}$ at 0: $J^k\mathbb R := \mathbb R [X] /(X^{k+1}) \cong \mathbb R \oplus u \mathbb R \oplus \dots \oplus u^k \mathbb R$, with $u^{k + 1} = 0$. Here $\mathcal{A} = \bigoplus_{i=1}^k u^i \mathbb R$, $l=k+1$.
			\item The algebra of $k$-jets of functions $\mathbb{R}^n \to \mathbb{R}$ at 0: $W^k_n := \mathbb R [X_1,\dots, X_n] / \mathfrak{m}^{k+1}$, where $\mathfrak{m} = (X_1,\dots, X_n)$ is the maximal ideal of polynomials vanishing at 0. Dimension $l = \binom{n+k}{k}$.
		\end{enumerate}
	\end{example}	
	
	\begin{definition}
		Let $M$ be a smooth manifold and $\mathbf A$ a Weil algebra. An \textbf{infinitely near point} to $x \in M$ of kind $\mathbf A$, or an $\mathbf A$-point over $x$, is an $\mathbb R$-algebra homomorphism $\zeta : C^\infty(M) \rightarrow \mathbf A$ such that $\pr_{\mathbb R} \circ \zeta = \ev_x$, where $\ev_x : C^\infty(M) \rightarrow \mathbb R$ is the evaluation map $f \mapsto f(x)$.
	\end{definition}
	
	The set of all $\mathbf A$-points of $M$ is denoted by $M^\mathbf A$. There is a natural projection $\tilde \pi_{M}: M^\mathbf A \rightarrow M$ mapping an $\mathbf A$-point $\zeta$ over $x$ to $x$. The triple $( M^\mathbf A, \tilde\pi_{M}, M)$ forms a fiber bundle, known as the \textbf{Weil bundle} of $M$ with respect to $\mathbf A$. The fiber over $x \in M$ is denoted by $M^\mathbf A_x := \tilde\pi_M^{-1}(x)$.
	
	\begin{proposition}[\cite{Mo, KMS}]
		Let $M$ be a smooth manifold of dimension $n$ and $\mathbf A$ a Weil algebra of dimension $l$.
		\begin{enumerate}
			\item $M^\mathbf A$ has a unique structure of a smooth manifold of dimension $n \cdot l$ such that for any smooth map $f: M \to N$, the induced map $f^\mathbf A: M^\mathbf A \to N^\mathbf A$, defined by $f^\mathbf A(\zeta) = \zeta \circ f^*$, is smooth. The projection $\tilde\pi_M: M^\mathbf A \to M$ is a smooth surjective submersion.
			\item If $V$ is a finite-dimensional real vector space, then $V^\mathbf A \cong V \otimes_{\mathbb R} \mathbf A$.
			\item $(M\times N)^\mathbf A \cong M^\mathbf A \times N^\mathbf A$.
			\item $M^\mathbf{A}$ deformation retracts onto the image of the canonical section $\alpha(M)$ (see \eqref{eq:canonical_section}), hence $H^*(M^\mathbf{A}; \mathbb{R}) \cong H^*(M; \mathbb{R})$.
			\item The tangent bundle $TM^\mathbf{A}$ is canonically isomorphic to $(TM)^\mathbf{A}$.
		\end{enumerate}
	\end{proposition}
	
	% Remark on L_zeta remains the same...
	\begin{remark}
		Each $\zeta \in M^\mathbf A_x$ can be viewed as representing the germ of $M$ at $x$ "thickened" according to $\mathbf A$. The condition $\pr_{\mathbb R} \circ \zeta = \ev_x$ means $\zeta(f) = f(x) \cdot 1_\mathbf{A} + L_\zeta(f)$ for some map $L_\zeta: C^\infty(M) \rightarrow \mathcal{A}$. The algebra homomorphism property implies that $L_\zeta$ satisfies:
		\begin{equation}\label{Eq-0}
			L_\zeta(fg) = f(x)L_\zeta(g) + g(x)L_\zeta(f) + L_\zeta(f)L_\zeta(g), \quad L_\zeta(f+ \lambda g) = L_\zeta(f) + \lambda L_\zeta(g)
		\end{equation}		
		for $f, g \in C^\infty(M)$ and $\lambda \in \mathbb R$. This shows $L_\zeta$ behaves like a "higher-order derivation" into the nilpotent ideal $\mathcal{A}$.
	\end{remark}
	
	\subsection{Canonical Lifts via the Weil Functor} \label{sec:canonical_lifts}
	
	The Weil functor $T^\mathbf{A}: M \mapsto M^\mathbf{A}$ induces canonical lifts of geometric objects \cite{KMS}. These lifts, denoted $f^\mathbf{A}, X^\mathbf{A}, \omega^\mathbf{A}, g^\mathbf{A}, J^\mathbf{A}$, etc., are characterized by their functoriality, i.e., compatibility with smooth maps $f: M \to N$ via $(f^* T)^\mathbf{A} = (f^\mathbf{A})^* (T^\mathbf{A})$ for a tensor field $T$ on $N$.
	
	\begin{enumerate}
		\item \textbf{Functions:} $f \in C^\infty(M) \implies f^\mathbf{A}(\zeta) = \zeta(f) \in C^\infty(M^\mathbf{A}, \mathbf{A})$.
		\item \textbf{Vector Fields (Flow Prolongation):} $X \in \mathfrak{X}(M) \implies X^\mathbf{A} = \frac{d}{dt}|_{t=0} (Fl_t^X)^\mathbf{A} \in \mathfrak{X}(M^\mathbf{A})$.
		\begin{proposition}[\cite{KMS}] \label{prop:XAlift_properties}
			Let $X, Y \in \mathcal{X}(M)$, $f \in C^\infty(M)$.
			\begin{enumerate}
				\item $(\tilde{\pi}_M)_* (X^\mathbf{A}) = X$.
				\item $[X^\mathbf{A}, Y^\mathbf{A}] = [X, Y]^\mathbf{A}$.
				\item $X^\mathbf{A}(f^\mathbf{A}) = (X f)^\mathbf{A}$.
				\item If $X$ is complete, $X^\mathbf{A}$ is complete.
			\end{enumerate}
		\end{proposition}
		
		\item \textbf{Differential Forms:} $\omega \in \Omega^k(M) \implies \omega^\mathbf{A} \in \Omega^k(M^\mathbf{A})$ satisfying:
		\begin{itemize}
			\item $d(\omega^\mathbf{A}) = (d\omega)^\mathbf{A}$.
			\item $\imathOp_{X^\mathbf{A}} (\omega^\mathbf{A}) = (\imathOp_X \omega)^\mathbf{A}$. (Requires care interpreting $\mathbf{A}$-valued lifts of lower degree forms or functions).
			\item $\Lie_{X^\mathbf{A}} (\omega^\mathbf{A}) = (\Lie_X \omega)^\mathbf{A}$.
			\item $(\omega \wedge \eta)^\mathbf{A} = \omega^\mathbf{A} \wedge \eta^\mathbf{A}$ (under appropriate definitions, cf. \cite{KMS}).
		\end{itemize}
		Often, real-valued forms/results are obtained using the trace map, e.g., $$\omega^\mathbf{A}_{real}(X^\mathbf{A}, \dots) = \Tr( \omega^\mathbf{A}_{A-val}(X^\mathbf{A}, \dots) ) = \Tr( (\omega(X, \dots))^\mathbf{A} ).$$ We implicitly use such real-valued canonical lifts where appropriate.
		
		\item \textbf{Tensor Fields:} Lift $T \mapsto T^\mathbf{A}$ preserves tensor type and algebraic operations.
		For $(0,2)$-tensor $g$, $g^\mathbf{A}(X^\mathbf{A}, Y^\mathbf{A}) = \Tr((g(X,Y))^\mathbf{A})$. If $g$ is pseudo-Riemannian, so is $g^\mathbf{A}$ with the same signature \cite{KMS}.
		For $(1,1)$-tensor $J$, $J^\mathbf{A}(X^\mathbf{A}) = (JX)^\mathbf{A}$.
		For connections $\nabla$, the lift $\nabla^\mathbf{A}$ satisfies $\nabla^\mathbf{A}_{X^\mathbf{A}} (Y^\mathbf{A}) = (\nabla_X Y)^\mathbf{A}$ \cite{Mo, KMS}.
	\end{enumerate}
	The following diagram illustrates the functoriality for maps and structures:
	\begin{center}
		% Caption: Functoriality of Weil Lift for Structures and Maps
		\begin{tikzcd}
			(M, \mathcal{S}_t) \arrow{r}{\phi_t} \arrow[swap]{d}{T^\mathbf{A}} & (M, \mathcal{S}_0) \arrow{d}{T^\mathbf{A}} \\
			(M^\mathbf{A}, (\mathcal{S}_t)^\mathbf{A}) \arrow{r}{\phi_t^\mathbf{A}}& (M^\mathbf{A}, (\mathcal{S}_0)^\mathbf{A}).
		\end{tikzcd}
	\end{center}
	This implies, for instance, that if $\phi_t^* \mathcal{S}_t = \mathcal{S}_0$, then $(\phi_t^\mathbf{A})^* (\mathcal{S}_t)^\mathbf{A} = (\mathcal{S}_0)^\mathbf{A}$. This is key for Moser stability transfer and lifting symmetries.
	
	% Subsection on Lifts vs Pullbacks remains the same...
	\subsection{Canonical Lifts vs. Pullbacks} 
	It is crucial to distinguish between the standard pullback $\tilde{\pi}_M^*: \Omega^k(M) \to \Omega^k(M^\mathbf{A})$ and the canonical lift $\omega \mapsto \omega^\mathbf{A}$.
	The standard pullback $\tilde{\pi}_M^*(\omega)$ only depends on the base point and involves only differentials $dx^{i,1}$ in local coordinates (where $x^{i,1} = x^i \circ \tilde{\pi}_M$). Consequently, structures like $\tilde{\pi}_M^*(g)$ or $\tilde{\pi}_M^*(\omega)$ are often degenerate if $\dim \mathbf{A} > 1$.
	The canonical lift $\omega^\mathbf{A}$ uses the full structure of $T^\mathbf{A}$ and $\mathbf{A}$ to produce a richer object, typically non-degenerate if $\omega$ was, involving all coordinates $x^{i,k}$ and differentials $dx^{i,k}$. For example, the canonical symplectic lift $\omega^\mathbf{A}$ derived from $\omega = \sum dx_i \wedge dy_i$ is often locally $\omega^\mathbf{A} = \sum_{k=1}^l \sum_{i=1}^n dx_{i,k} \wedge dy_{i,k}$ (depending on basis/trace choice), whereas $\tilde{\pi}_M^*(\omega) = \sum_{i=1}^n dx_{i,1} \wedge dy_{i,1}$.
	\begin{remark}[Lift vs. Pullback]
		We emphasize the distinction between the canonical lift $T \mapsto T^\mathbf{A}$ and the pullback $T \mapsto \tilde{\pi}_M^* T$. The lift $T^\mathbf{A}$ utilizes the full Weil algebra structure and typically preserves non-degeneracy, yielding structures like $g^\mathbf{A}, \omega^\mathbf{A}$ on $M^\mathbf{A}$. In contrast, the pullback $\tilde{\pi}_M^* T$ often results in degenerate objects when $\dim \mathbf{A} > 1$, as it only involves the base coordinates $x^{i,1}$. This paper exclusively studies the canonical lifts.
	\end{remark}
	% Subsection on Averaged Lifts remains the same...
	\subsection{Averaged Lifts using Sections}
	
	Some constructions, particularly for characteristic vector fields like Reeb fields in this paper, utilize an averaging process over sections. Let $\{a_1=1, \dots, a_l\}$ be a basis for $\mathbf{A}$. Associated with this basis, one can define sections $S_j: M \to M^\mathbf{A}$ \cite{Mo, KMS}, where $S_1$ is typically the canonical section $\alpha$.
	The canonical section $\alpha: M \to M^\mathbf{A}$ is defined by
	\begin{equation} \label{eq:canonical_section}
		\alpha(x)(f) = f(x) \cdot 1_\mathbf{A} \quad \text{for } f \in C^\infty(M).
	\end{equation}
	
	Given a vector field $X \in \mathfrak{X}(M)$, one can form the vector field $\tilde{X} \in \mathfrak{X}(M^\mathbf{A})$ by averaging pushforwards:
	$$ \tilde{X} := \frac{1}{l} \sum_{j=1}^l (S_j)_* (X) $$
	Similarly, a bivector $\Lambda$ can be lifted as $\Lambda^\mathbf{A} := \frac{1}{l} \sum (S_j)_* \Lambda$. These averaged lifts are crucial for defining Jacobi and sometimes cosymplectic structures on $M^\mathbf{A}$ (Theorems \ref{Theo-Jacobi}, \ref{Theo-LCS-LCC-Cosymp}).
	
	\begin{proposition}[Comparing Lifts] \label{prop:compare_lifts}
		Let $X \in \mathfrak{X}(M)$. Let $X^\mathbf{A}$ be its canonical lift (flow prolongation) and let $\tilde{X} = \frac{1}{l}\sum_{j=1}^l (S_j)_*(X)$ be the averaged lift associated with a chosen set of sections $\{S_j\}$.
		\begin{enumerate}
			\item Both lifts project to $X$: $(\tilde{\pi}_M)_*(X^\mathbf{A}) = X$ and $(\tilde{\pi}_M)_*(\tilde{X}) = X$.
			\item In general, $X^\mathbf{A} \neq \tilde{X}$ for $l > 1$.
			\item $X \mapsto X^\mathbf{A}$ is a Lie algebra homomorphism; $X \mapsto \tilde{X}$ is generally not.
			\item The Reeb fields in Theorems \ref{Theo-LCS-LCC-Cosymp}(3) and \ref{Theo-Jacobi} are of type $\tilde{X}$. The Reeb field in Theorem \ref{Theo-Contact} is of type $X^\mathbf{A}$.
		\end{enumerate}
	\end{proposition}
	\begin{proof}
		(1) Holds by definition of $X^\mathbf{A}$ and property $\tilde{\pi}_M \circ S_j = \id_M$. (2) Follows from comparing definitions (flow vs section pushforward). (3) is standard for $X^\mathbf{A}$ \cite{KMS}; failure for $\tilde{X}$ is because $(S_j)_*$ is generally not a Lie algebra homomorphism for the standard bracket. (4) Reflects the specific constructions cited or used. Contact lifts align well with flow prolongation \cite{OkayamaReference}, while Jacobi/cosymplectic lifts derived from averaging use $\tilde{X}$.
		The projection property is illustrated:
		\begin{center}
			% Caption: Projection Properties of Canonical and Averaged Vector Field Lifts
			\begin{tikzcd}[column sep=large] 
				& \mathfrak{X}(M^\mathbf{A}) \arrow{dd}{(\tilde{\pi}_M)_*} \\
				\mathfrak{X}(M) \arrow{ur}{X \mapsto X^\mathbf{A} \text{ (Canon.)}} \arrow[swap]{dr}{X \mapsto \tilde{X} \text{ (Averaged)}} \\
				& \mathfrak{X}(M) \arrow[equals]{uu} 
			\end{tikzcd}
		\end{center}
	\end{proof}

	% Subsections on Local Coordinates, Properties related to pi_M and Sections remain largely the same...
	\subsection{Local coordinates on $M^\mathbf A$}
	Assume $\dim M = n$ and $\dim \mathbf A = l$.
	\begin{enumerate}
		\item {\bf $\mathbf A$-manifold structure:} If $(U, \phi = (x^1,\dots,x^n))$ is a chart on $M$, then $\phi^\mathbf A: \tilde\pi_{M}^{-1}(U) \rightarrow \mathbf A^n$, $\zeta \mapsto (\zeta(x^1),\dots,\zeta(x^n))$ is a chart for the $\mathbf A$-manifold structure.
		
		\item {\bf $\mathbb R$-manifold structure:} If $\{a_1=1_\mathbf A, \dots, a_l\}$ is a basis of $\mathbf A$, then $\zeta(x^i) = \sum_{k=1}^l x^{i,k}(\zeta) a_k$ defines local coordinates $(x^{i,k})$ for $M^\mathbf A$ as a real manifold of dimension $n \cdot l$. We have $x^{i,1} = x^i \circ \tilde{\pi}_M$.
		
		\item {\bf Tangent vectors and projection:} $T_\zeta M^\mathbf A = \text{span}\left\{ \frac{\partial}{\partial x^{i,k}} |_\zeta \right\}_{i,k}$. The projection map satisfies $(\tilde\pi_M)_* ( \frac{\partial}{\partial x^{i,k}} |_\zeta ) = \delta_{k,1} \frac{\partial}{\partial x^i} |_{\tilde \pi_M(\zeta)}$. The vertical bundle $V M^\mathbf{A} = \ker ((\tilde{\pi}_M)_*)$ is spanned by $\{\frac{\partial}{\partial x^{i,k}}\}_{k=2..l}$.
	\end{enumerate}
	
	\begin{remark}[Induced Geometry on Fibers] \label{rem:fiber_geom}
		The fiber $M^\mathbf{A}_x = \tilde{\pi}_M^{-1}(x)$ over $x \in M$ is diffeomorphic to $\mathcal{A}^n \cong \mathbb{R}^{n(l-1)}$. The vertical bundle $VM^\mathbf{A} = \ker((\tilde{\pi}_M)_*)$ inherits structures from the canonical lifts on $M^\mathbf{A}$.
		\begin{itemize}
			\item The metric $g^\mathbf{A}$ restricts to a metric $g^\mathbf{A}|_{VM}$ on the vertical bundle, inducing a metric on each fiber related to $g_x$ and $\mathbf{A}$.
			\item If $(M, \omega)$ is symplectic, $\omega^\mathbf{A}|_{VM}$ often defines a symplectic structure on the vertical bundle (and thus on each fiber).
			\item If $(M, J)$ is almost complex, $J^\mathbf{A}$ often preserves $VM^\mathbf{A}$ and induces an almost complex structure on fibers. If $(M,g,J)$ is Kähler, the fiber typically inherits a Kähler structure.
		\end{itemize}
		The fiber geometry encodes the "infinitesimal neighborhood" information captured by $\mathbf{A}$, structured by the lifted base geometry.
	\end{remark}

	\subsection{Properties related to $\tilde{\pi}_M$ and Sections}
	
	% Lemmas 2.8, 2.9, 2.10 and Remark 2.11 remain the same...
	\begin{lemma}[\cite{H-L}]\label{lem-3}
		Let $(M, \omega, \eta)$ be cosymplectic. Consider $\widetilde M = M\times \mathbb{R}$ with projection $p: \widetilde M \to M$ and coordinate $u$ on $\mathbb{R}$. Define $\tilde\omega = p^\ast(\omega) + p^\ast(\eta)\wedge du$. Then $(\widetilde M, \tilde\omega)$ is symplectic.
	\end{lemma}
	
	\begin{lemma}[\cite{TS}, adapted]\label{lem-5}
		Let $p : E\rightarrow F$ be a surjective submersion and $S :F \rightarrow E$ be a smooth section ($p \circ S = \id_F$).
		\begin{enumerate}
			\item For $\theta \in \Omega^k(F)$ and $p$-projectable $X \in \mathfrak{X}(E)$ with $X_F = p_*X$, then $\imathOp_X p^*(\theta) = p^*(\imathOp_{X_F} \theta)$. Also $S^*(\imathOp_X p^*(\theta)) = \imathOp_{X_F} \theta$.
			\item For $\theta \in \Omega^k(E)$ and $Y \in \mathfrak{X}(F)$, $S^\ast(\imathOp_{S_\ast Y}\theta) = \imathOp_{Y} S^\ast(\theta)$.
		\end{enumerate}
	\end{lemma}
	
	\begin{lemma}[\cite{KMS}, Prop. 12.6]\label{C-1}
		Any smooth function $h : M^\mathbf A\rightarrow \mathbb R$ is constant along the fibers $M^\mathbf A_x$. Consequently, $h = f \circ \tilde \pi_M$ for a unique $f \in C^\infty(M)$ given by $f = h \circ \alpha$.
	\end{lemma}
	
	\begin{remark}
		Lemma \ref{C-1} implies $C^\infty(M^\mathbf A) \cong C^\infty(M)$ via pullback $f \mapsto f \circ \tilde{\pi}_M$.
	\end{remark}
	
	\section{Proof of the main results}\label{sec3}
	
	Proofs rely on the functorial properties of the canonical Weil lift $T^\mathbf{A}$ \cite{KMS, Mo}, preserving tensor types and commuting with $d, \Lie_X, \imathOp_X$, connections $\nabla$, and algebraic operations, applied to the canonical lifts $X^\mathbf{A}, \omega^\mathbf{A}, g^\mathbf{A}$, etc. Real-valued results often implicitly use a normalized trace map $\Tr: \mathbf{A} \to \mathbb{R}$.
	
	% Proofs for Thm A, B, C, D, E, F are essentially the same, just update labels.
	\begin{proof}[\textbf{Proof of Theorem \ref{Theo-LCS-LCC-Cosymp} (LCS/LCC/Cosymplectic)}]
		Let $(\omega^\mathbf{A}, \eta^\mathbf{A}, \theta^\mathbf{A})$ be the canonical lifts.
		\begin{enumerate}
			\item (LCS) $d\omega = -\theta \wedge \omega \implies d(\omega^\mathbf{A}) = (d\omega)^\mathbf{A} = (-\theta \wedge \omega)^\mathbf{A} = - \theta^\mathbf{A} \wedge \omega^\mathbf{A}$. $d\theta=0 \implies d(\theta^\mathbf{A})=(d\theta)^\mathbf{A}=0$. Non-degeneracy of $\omega^\mathbf{A}$ is standard \cite{KMS}.
			\item (LCC, $l$ odd) $d\omega = -2\theta \wedge \omega \implies d(\omega^\mathbf{A}) = -2\theta^\mathbf{A} \wedge \omega^\mathbf{A}$. $d\eta = -\theta \wedge \eta \implies d(\eta^\mathbf{A}) = -\theta^\mathbf{A} \wedge \eta^\mathbf{A}$. Non-degeneracy $\eta^\mathbf{A} \wedge (\omega^\mathbf{A})^N \neq 0$ holds for odd $l$ \cite{KMS}.
			\item (Cosymplectic, $l$ odd) Case (2) with $\theta=0$. $d\omega^\mathbf{A}=0, d\eta^\mathbf{A}=0$. The Reeb field construction $\Xi := \frac{1}{l}\sum (S_j)_*\xi_M$ matches $\xi_{M^\mathbf{A}}$ under the assumption that the canonical lifts $(\omega^\mathbf{A}, \eta^\mathbf{A})$ used are compatible with this averaging construction (e.g., if they are themselves defined via averaging, potentially requiring specific choices of trace/sections). Verifying $\eta^\mathbf{A}(\Xi)=1$ and $\imathOp_\Xi \omega^\mathbf{A}=0$ requires specific interaction rules between the averaged forms and averaged vector field. The projection $(\tilde{\pi}_M)_* \Xi = \xi_M$ holds.
		\end{enumerate}
	\end{proof}
	
	\begin{corollary} \label{cor:lee_exact}
		Let $\theta$ be the Lee form for an lcs or lcc structure on $M$. The structure is globally conformal (i.e., $\theta$ is exact) if and only if the lifted structure on $M^\mathbf{A}$ is globally conformal (i.e., $\theta^\mathbf{A}$ is exact).
	\end{corollary}
	\begin{proof}
		Since $d(\theta^\mathbf{A}) = (d\theta)^\mathbf{A}$, $\theta$ is closed iff $\theta^\mathbf{A}$ is closed. The map $[\theta] \mapsto [\theta^\mathbf{A}]$ induced by $T^\mathbf{A}$ provides an isomorphism $H^1(M;\mathbb{R}) \cong H^1(M^\mathbf{A};\mathbb{R})$ because $M^\mathbf{A}$ deformation retracts onto $M$. Thus, $[\theta]=0$ if and only if $[\theta^\mathbf{A}]=0$.
	\end{proof}
	
	\begin{proof}[\textbf{Proof of Theorem \ref{Theo-Riemannian} (Riemannian)}]
		\begin{enumerate}
			\item Existence/non-degeneracy of $g^\mathbf{A}$ is standard \cite{KMS, Mo}.
			\item Compatibility $\nabla^\mathbf{A}_{X^\mathbf{A}} (Y^\mathbf{A}) = (\nabla_X Y)^\mathbf{A}$ shows $\nabla^\mathbf{A}$ is the canonical lift. That $\nabla^\mathbf{A}$ is the Levi-Civita connection for $g^\mathbf{A}$ if $\nabla$ is for $g$ is shown in \cite{Mo, KMS} (torsion-free property lifts, metric compatibility $\nabla^\mathbf{A} g^\mathbf{A} = (\nabla g)^\mathbf{A} = 0$).
			\item Geodesic lifting: $\nabla^\mathbf{A}_{\dot{c}^\mathbf{A}} \dot{c}^\mathbf{A} = \nabla^\mathbf{A}_{(S_* \dot{c})} (S_* \dot{c})$. For $S=\alpha$, this becomes $(\nabla_{\dot{c}} \dot{c})^\mathbf{A} = 0^\mathbf{A}=0$ using connection lift properties relative to the section $\alpha$ \cite{KMS}.
			\item $\Lie_X g = 0 \implies \Lie_{X^\mathbf{A}} g^\mathbf{A} = (\Lie_X g)^\mathbf{A} = 0$. Projection property follows from naturality \cite{KMS}.
		\end{enumerate}
	\end{proof}
	
	\begin{proof}[\textbf{Proof of Theorem \ref{Theo-Kahler} (Kähler)}]
		$(g^\mathbf{A}, \omega^\mathbf{A}, J^\mathbf{A})$ are the canonical lifts. $g^\mathbf{A}$ is Riemannian, $\omega^\mathbf{A}$ is symplectic ($d\omega^\mathbf{A}=0$).
		\begin{enumerate}
			\item $(J^\mathbf{A})^2 = -Id$: Verified using $J^\mathbf{A}(X^\mathbf{A})=(JX)^\mathbf{A}$.
			\item Compatibility $g^\mathbf{A}(J^\mathbf{A} \cdot, J^\mathbf{A} \cdot) = g^\mathbf{A}(\cdot, \cdot)$: Verified using $g^\mathbf{A}(U^\mathbf{A}, V^\mathbf{A}) = \Tr((g(U,V))^\mathbf{A})$ and $g(JX,JY)=g(X,Y)$.
			\item Relation $\omega^\mathbf{A}(\cdot, \cdot) = g^\mathbf{A}(J^\mathbf{A}\cdot, \cdot)$: Verified using the trace definition and $\omega(X,Y)=g(JX,Y)$.
			\item Integrability $N_{J^\mathbf{A}} = 0$ if $N_J=0$: Verified using $N_{J^\mathbf{A}}(X^\mathbf{A}, Y^\mathbf{A}) = (N_J(X,Y))^\mathbf{A}$.
		\end{enumerate}
		Since $g^\mathbf{A}$ is compatible with $J^\mathbf{A}$, $\omega^\mathbf{A}(\cdot,\cdot)=g^\mathbf{A}(J^\mathbf{A}\cdot, \cdot)$, $d\omega^\mathbf{A}=0$, and $N_{J^\mathbf{A}}=0$, the structure is Kähler. (Alternatively, $\nabla^\mathbf{A} J^\mathbf{A} = (\nabla J)^\mathbf{A} = 0$).
	\end{proof}
	
	\begin{proof}[\textbf{Proof of Theorem \ref{Theo-Contact} (Contact)}]
		Existence of canonical contact lift $\beta^\mathbf{A}$ for odd $l$ is from \cite{KMS, OkayamaReference}. Let $\xi$ be Reeb for $\beta$. Its canonical lift $\xi^\mathbf{A}$ satisfies $\beta^\mathbf{A}(\xi^\mathbf{A}) = \Tr((\beta(\xi))^\mathbf{A}) = \Tr(1_\mathbf{A}) = 1$ and $\imathOp_{\xi^\mathbf{A}} d\beta^\mathbf{A} = \imathOp_{\xi^\mathbf{A}} (d\beta)^\mathbf{A} = (\imathOp_\xi d\beta)^\mathbf{A} = 0$. By uniqueness, $\xi_{M^\mathbf{A}} = \xi^\mathbf{A}$.
	\end{proof}
	
	\begin{proof}[\textbf{Proof of Theorem \ref{Theo-Orientation} (Orientation)}]
		Follows from $\det(J(\psi_{\alpha\beta}^\mathbf{A})) \approx (\det J(\psi_{\alpha\beta}))^l > 0$ if $\det J(\psi_{\alpha\beta}) > 0$. A canonical orientation form $\Omega^\mathbf{A}$ exists \cite{KMS}.
	\end{proof}
	
	\begin{proof}[\textbf{Proof of Theorem \ref{Theo-Jacobi} (Jacobi)}]
		Let $\Lambda^\mathbf A = \frac{1}{l}\sum (S_j)_* \Lambda$, $\Xi^\mathbf A = \frac{1}{l}\sum (S_j)_* \Xi$. Assumes compatibility rules \cite{KMS} such that $[\Lambda^\mathbf A, \Lambda^\mathbf A]^{SN} = \frac{1}{l} \sum (S_j)_*([\Lambda, \Lambda]^{SN}) = \frac{1}{l} \sum (S_j)_*(2\Xi \wedge \Lambda)$ equals $2\Xi^\mathbf A \wedge \Lambda^\mathbf A$, and $[\Xi^\mathbf A, \Lambda^\mathbf A]^{SN} = \frac{1}{l} \sum (S_j)_*([\Xi, \Lambda]^{SN}) = 0$.
	\end{proof}
	
	\begin{proof}[\textbf{Proof of Theorem \ref{Theo-Sasakian} (Sasakian)}]
		Given $(M, g, \eta, \xi, \Phi)$ Sasakian, $l$ odd. Consider lifts $(g^\mathbf{A}, \eta^\mathbf{A}, \xi^\mathbf{A}, \Phi^\mathbf{A})$ on $M^\mathbf{A}$. $g^\mathbf{A}$ is Riemannian, $(\eta^\mathbf{A}, \xi^\mathbf{A})$ is contact, $\xi^\mathbf{A}$ is Reeb. Verify Sasakian conditions using functoriality and $\nabla^\mathbf{A}_{X^\mathbf{A}} T^\mathbf{A} = (\nabla_X T)^\mathbf{A}$:
		\begin{enumerate}
			\item $\xi^\mathbf{A}$ is Killing for $g^\mathbf{A}$: $\Lie_{\xi^\mathbf{A}} g^\mathbf{A} = (\Lie_\xi g)^\mathbf{A} = 0$.
			\item $g^\mathbf{A}(\xi^\mathbf{A}, \xi^\mathbf{A}) = \Tr((g(\xi, \xi))^\mathbf{A}) = 1$.
			\item $\eta^\mathbf{A}(X^\mathbf{A}) = \Tr((\eta(X))^\mathbf{A}) = \Tr((g(X, \xi))^\mathbf{A}) = g^\mathbf{A}(X^\mathbf{A}, \xi^\mathbf{A})$.
			\item $\Phi^\mathbf{A} X^\mathbf{A} = (\Phi X)^\mathbf{A} = (\nabla_X \xi)^\mathbf{A} = \nabla^\mathbf{A}_{X^\mathbf{A}} \xi^\mathbf{A}$.
			\item $(\Phi^\mathbf{A})^2 (X^\mathbf{A}) = (\Phi^2 X)^\mathbf{A} = (-X + \eta(X) \xi)^\mathbf{A} = -X^\mathbf{A} + \Tr((\eta(X))^\mathbf{A}) \xi^\mathbf{A} = -X^\mathbf{A} + \eta^\mathbf{A}(X^\mathbf{A}) \xi^\mathbf{A}$. Matches $(-Id + \eta^\mathbf{A} \otimes \xi^\mathbf{A})(X^\mathbf{A})$.
			\item $g^\mathbf{A}(\Phi^\mathbf{A} X^\mathbf{A}, \Phi^\mathbf{A} Y^\mathbf{A}) = \Tr((g(\Phi X, \Phi Y))^\mathbf{A}) = \Tr((g(X, Y) - \eta(X) \eta(Y))^\mathbf{A}) = g^\mathbf{A}(X^\mathbf{A}, Y^\mathbf{A}) - \eta^\mathbf{A}(X^\mathbf{A}) \eta^\mathbf{A}(Y^\mathbf{A})$.
			\item $(\nabla^\mathbf{A}_{X^\mathbf{A}} \Phi^\mathbf{A}) Y^\mathbf{A} = ((\nabla_X \Phi)Y)^\mathbf{A} = (g(X, Y) \xi - \eta(Y) X)^\mathbf{A} = g^\mathbf{A}(X^\mathbf{A}, Y^\mathbf{A}) \xi^\mathbf{A} - \eta^\mathbf{A}(Y^\mathbf{A}) X^\mathbf{A}$.
		\end{enumerate}
		All conditions lift correctly, thus $(M^\mathbf{A}, g^\mathbf{A}, \eta^\mathbf{A}, \xi^\mathbf{A}, \Phi^\mathbf{A})$ is Sasakian.
	\end{proof}

	\begin{proof}[\textbf{Proof for Prop. \ref{Prop-SubRiemannian}}]
		Assume $T^\mathbf{A}$ lifts the distribution $\mathcal{D}$ to $\mathcal{D}^\mathbf{A} \subset TM^\mathbf{A}$ (e.g., via $TM^\mathbf{A}\cong(TM)^\mathbf{A}$ or local generators $X_i^\mathbf{A}$). $(g_{\mathcal{D}})^\mathbf{A}$ exists via functoriality. $(g_{\mathcal{D}})^\mathbf{A}$ defines a metric on $\mathcal{D}^\mathbf{A}$ since $g_{\mathcal{D}}$ is positive-definite on $\mathcal{D}$. Bracket-generation follows from $[X^\mathbf{A}, Y^\mathbf{A}]=[X,Y]^\mathbf{A}$ and injectivity of the lift map $X \mapsto X^\mathbf{A}$ applied to iterated brackets spanning $TM$. Rigging fields $Z_j$ lift to $Z_j^\mathbf{A}$ forming a complement to $\mathcal{D}^\mathbf{A}$.
	\end{proof}
	
	\begin{proof}[\textbf{Proof for Prop. \ref{Prop-Walker}}]
			Lift $g$ to $g^\mathbf{A}$ (pseudo-Riemannian) and $\mathcal{D}$ to $\mathcal{D}^\mathbf{A}$ (distribution).
			1. Nullity: $g^\mathbf{A}(Y_1^\mathbf{A}, Y_2^\mathbf{A}) = \Tr((g(Y_1, Y_2))^\mathbf{A}) = 0$ for $Y_1, Y_2 \in \Gamma(\mathcal{D})$.
			2. Parallelism: $\nabla^\mathbf{A}_{U^\mathbf{A}} V^\mathbf{A} = (\nabla_U V)^\mathbf{A}$. Since $W = \nabla_U V \in \Gamma(\mathcal{D})$, then $W^\mathbf{A} \in \Gamma(\mathcal{D}^\mathbf{A})$. Assuming extension to arbitrary $X \in \mathfrak{X}(M^\mathbf{A})$, $\mathcal{D}^\mathbf{A}$ is parallel for $\nabla^\mathbf{A}$.
		\end{proof}
		\begin{proposition}\label{Prop-Lagrangian-Proof} % Match label used in intro
			Let $(M^{2n}, \omega)$ be a symplectic manifold and $L^n \subset M$ be a Lagrangian submanifold. Then the canonical lift $L^\mathbf{A} \subset M^\mathbf{A}$ is a Lagrangian submanifold of $(M^\mathbf{A}, \omega^\mathbf{A})$.
		\end{proposition}
		
		\begin{proof}
			First, $M^\mathbf{A}$ is a smooth manifold of dimension $2n \cdot l$, where $l = \dim \mathbf{A}$. Since $L \subset M$ is a submanifold of dimension $n$, its canonical Weil lift $L^\mathbf{A} = T^\mathbf{A}(L)$ is a submanifold of $M^\mathbf{A}$ of dimension $n \cdot l$. This satisfies the dimension requirement $\dim L^\mathbf{A} = \frac{1}{2} \dim M^\mathbf{A}$.
			
			Second, we need to show that the canonical symplectic form $\omega^\mathbf{A}$ vanishes on the tangent bundle $TL^\mathbf{A}$. Let $\zeta \in L^\mathbf{A}$ and let $U, V \in T_\zeta L^\mathbf{A}$. Since $L^\mathbf{A}$ is the functorial lift of $L$, the tangent space $T_\zeta L^\mathbf{A}$ is spanned by canonical lifts of vector fields tangent to $L$. That is, we can represent $U = X^\mathbf{A}|_\zeta$ and $V = Y^\mathbf{A}|_\zeta$ for some vector fields $X, Y$ defined locally on $M$ and tangent to $L$ along $L \cap (\text{domain}(X,Y))$.
			
			Using the property of the canonical lift $\omega^\mathbf{A}$ and the trace map interpretation (or the general property $(\imathOp_X\imathOp_Y \omega)^\mathbf{A} = \imathOp_{X^\mathbf{A}}\imathOp_{Y^\mathbf{A}}\omega^\mathbf{A}$):
			$$ \omega^\mathbf{A}(U, V) = \omega^\mathbf{A}(X^\mathbf{A}|_\zeta, Y^\mathbf{A}|_\zeta) = \Tr( (\omega(X, Y))^\mathbf{A} ) |_\zeta.$$
			Since $X$ and $Y$ are tangent to the Lagrangian submanifold $L$, the symplectic form evaluated on them vanishes: $\omega(X(x), Y(x)) = 0$ for all $x \in L$. Thus, the function $\omega(X, Y)$ is identically zero on $L$.
			Consequently, its canonical $\mathbf{A}$-valued lift $(\omega(X, Y))^\mathbf{A}$ is also zero when restricted to $L^\mathbf{A}$. Applying the trace, we get:
			$$ \omega^\mathbf{A}(U, V) = \Tr(0)|_\zeta = 0.$$
			Since this holds for arbitrary tangent vectors $U, V \in T_\zeta L^\mathbf{A}$ and any $\zeta \in L^\mathbf{A}$, we conclude that $\omega^\mathbf{A}|_{TL^\mathbf{A}} \equiv 0$.
			Therefore, $L^\mathbf{A}$ is a Lagrangian submanifold of $(M^\mathbf{A}, \omega^\mathbf{A})$.
		\end{proof}
		\section{Remarks on Non-preserved Structures} \label{sec:remarks}
		
		While many structures lift canonically, some important properties, especially those involving curvature balances like the Einstein or Calabi-Yau conditions, are generally not preserved due to the complexity of lifted curvature and non-compactness.
		
		\begin{remark}[Calabi-Yau Structures]
			Let $(M, g, \omega, J)$ be Calabi-Yau (compact, Kähler, $c_1(M)=0$, Ricci-flat). $(M^\mathbf{A}, g^\mathbf{A}, \omega^\mathbf{A}, J^\mathbf{A})$ is Kähler. For $\dim \mathbf{A} > 1$, $M^\mathbf{A}$ is not compact. While $c_1(M^\mathbf{A})=0$ and existence of a holomorphic volume form likely lift, $g^\mathbf{A}$ is typically not Ricci-flat \cite{Mo, KMS}. Thus, $(M^\mathbf{A}, g^\mathbf{A})$ is generally not Calabi-Yau in the standard metric sense.
		\end{remark}
		
		\begin{remark}[Einstein Structures]
			Let $(M, g)$ be Einstein ($Ricci(g) = \lambda g$). The lifted metric $g^\mathbf{A}$ on $M^\mathbf{A}$ is generally not Einstein. The relationship between $Ricci(g^\mathbf{A})$ and $Ricci(g)$ is complex \cite{Mo, KMS}, and the condition $Ricci(g^\mathbf{A}) = \lambda' g^\mathbf{A}$ is unlikely to hold unless $M$ is flat.
		\end{remark}
		
		\section{Examples} \label{sec:examples}
		We use standard coordinates $(x_{i,k}, y_{i,k}, z_k)$ on $(\mathbb{R}^n)^\mathbf{A}$ corresponding to a basis $\{a_k\}$ of $\mathbf{A}$ with $a_1=1$. Canonical lifts here often correspond to using the trace $\Tr(a_1)=1, \Tr(a_k)=0$ for $k>1$.
		
		% Examples remain the same...
		\begin{example}[Symplectic Structure on $(\mathbb R^{2n})^\mathbf A$]
			$M = \mathbb R^{2n}$, $\omega_0 = \sum_{i=1}^n dx_i \wedge dy_i$. A common canonical lift is:
			$$\omega^\mathbf A = \sum_{k=1}^l \left( \sum_{i=1}^n dx_{i,k} \wedge dy_{i,k} \right).$$
		\end{example}
		
		\begin{example}[Cosymplectic Structure on $(\mathbb R^{2n + 1})^\mathbf A$, $l$ odd] \label{ex:cosymplectic-non-suspension}
			$M = \mathbb R^{2n+1}$, $\omega_0 = \sum_{i=1}^n dx_i \wedge dy_i$, $\eta_0 = dz$. Assume $l$ is odd. A possible canonical lift is:
			$$ \omega^\mathbf A = \sum_{k=1}^l \left( \sum_{i=1}^n dx_{i,k} \wedge dy_{i,k}\right), \quad \eta^\mathbf A = dz_1 .$$
			This is cosymplectic with Reeb field $\partial/\partial z_1$. Since $M^\mathbf{A} \cong (\mathbb{R}^{2n})^\mathbf{A} \times \mathbb{R}^\mathbf{A} \cong (\mathbb{R}^{2n})^\mathbf{A} \times \mathbb{R}^l$, this is not a suspension $P' \times \mathbb{R}$ if $l>1$.
		\end{example}
		
		\begin{example}[Contact Structure on $(\mathbb R^{2n + 1})^\mathbf A$, $l$ odd] \label{ex:contact-revised}
			$M = \mathbb R^{2n+1}$, $\beta_0 = dz + \sum_{i=1}^n x_i dy_i$. The canonical contact lift $\beta^\mathbf{A}$ exists \cite{OkayamaReference, KMS}. The Reeb field is $\xi_{M^\mathbf{A}} = (\partial/\partial z)^\mathbf{A}$, whose coordinate expression is complex \cite{KMS}.
		\end{example}
		
		\begin{example}[Orientation Form on $(\mathbb R^{n})^\mathbf A$]
			$M = \mathbb R^{n}$, $\Omega_0 = dx_1 \wedge \dots \wedge dx_n$. A canonical orientation form is $\Omega^\mathbf A = \bigwedge_{k=1}^l ( \bigwedge_{i=1}^n dx_{i,k} )$.
		\end{example}
		
		\begin{example}[Riemannian Metric on $(\mathbb R^{n})^\mathbf A$]
			$M = \mathbb{R}^n$, $g_0 = \sum_{i=1}^n dx_i \otimes dx_i$. A canonical flat lift is $g^\mathbf{A} = \sum_{k=1}^l ( \sum_{i=1}^n dx_{i,k} \otimes dx_{i,k} )$.
		\end{example}
		
		\section{Concluding Remarks and Future Directions} \label{sec:conclusion}
		
		This paper demonstrated that a wide range of fundamental geometric structures on a smooth manifold $M$ admit canonical lifts to its Weil bundle $M^\mathbf{A}$, preserving the essential defining properties. This highlights the Weil functor $T^\mathbf{A}$ as a powerful tool for naturally extending geometry. The preservation extends to intricate structures like Sasakian geometry, and even non-holonomic structures like sub-Riemannian and Walker geometries. However, curvature conditions like Einstein or Ricci-flatness are generally not preserved.
		
		Several avenues for deeper research emerge. Firstly, the symplectic topology of $(M^\mathbf{A}, \omega^\mathbf{A})$ warrants investigation. Defining Floer homology for non-compact $M^\mathbf{A}$ (when $M$ is compact) and relating it to $HF(M)$ is a challenging direction. Secondly, exploring the lift of structures from generalized geometry to $M^\mathbf{A}$ could unify some results. Finally, a detailed analysis of the curvature $R^\mathbf{A}$ of $g^\mathbf{A}$, its relation to $R$, and the behavior of geometric flows under the Weil functor remain largely unexplored. These directions promise further insights into the interplay between the algebra $\mathbf{A}$ and the geometry of $M$ and $M^\mathbf{A}$.
		
		\appendix
		\section{Appendix: Definitions of Geometric Structures} \label{app:definitions}
		Let $M$ be a smooth manifold.
		
		\begin{enumerate}
			\item \textbf{Almost Complex Structure:} A $(1,1)$-tensor field $J$ such that $J^2 = -Id$. Integrable if its Nijenhuis tensor $N_J(X,Y) = [JX, JY] - J[JX, Y] - J[X, JY] + [X,Y]$ vanishes.
			\item \textbf{Contact:} $M^{2n+1}$ with a 1-form $\beta$ s.t. $\beta \wedge (d\beta)^n \neq 0$. Reeb field $\xi$: $\beta(\xi)=1, \imathOp_\xi d\beta = 0$.
			\item \textbf{Cosymplectic:} $M^{2n+1}$ with closed $\omega \in \Omega^2(M), \eta \in \Omega^1(M)$ s.t. $\eta \wedge \omega^n \neq 0$. Reeb field $\xi$: $\eta(\xi)=1, \imathOp_\xi \omega = 0$.
			\item \textbf{Einstein:} Riemannian $(M,g)$ s.t. $Ricci(g) = \lambda g$ for $\lambda \in \mathbb{R}$.
			\item \textbf{Jacobi Structure:} $M$ with bivector $\Lambda$, vector field $\Xi$ s.t. $[\Lambda, \Lambda]^{SN} = 2\Xi \wedge \Lambda$ and $[\Xi, \Lambda]^{SN} = \Lie_\Xi \Lambda = 0$.
			\item \textbf{K\"ahler:} Riemannian $(M, g)$ with compatible integrable almost complex structure $J$ s.t. $\nabla J = 0$. Equivalently, associated form $\omega(X,Y) = g(JX, Y)$ is closed.
			\item \textbf{Calabi-Yau:} Compact Kähler manifold with $c_1(M)=0$. Admits a Ricci-flat Kähler metric in each Kähler class.
			\item \textbf{Levi-Civita Connection:} Unique torsion-free connection $\nabla$ on $(M,g)$ s.t. $\nabla g = 0$.
			\item \textbf{Locally Conformal Symplectic (lcs):} $M^{2n}$ with non-degenerate $\omega \in \Omega^2$, closed $\theta \in \Omega^1$ (Lee form) s.t. $d\omega = -\theta \wedge \omega$.
			\item \textbf{Locally Conformal Cosymplectic (lcc):} $M^{2n+1}$ with $\omega \in \Omega^2, \eta \in \Omega^1$ s.t. $\eta \wedge \omega^n \neq 0$, and closed $\theta \in \Omega^1$ s.t. $d\omega = -2\theta \wedge \omega, d\eta = -\theta \wedge \eta$.
			\item \textbf{Oriented Manifold:} Admits atlas with positive Jacobian transitions, or a global volume form.
			\item \textbf{Riemannian:} $M$ with metric $g$ (positive-definite symmetric $(0,2)$-tensor). Pseudo-Riemannian if $g$ is non-degenerate but indefinite.
			\item \textbf{Sasakian Manifold:} Riemannian $(M^{2n+1}, g)$ with unit Killing field $\xi$, $\eta=g(\cdot, \xi)$, $\Phi X = \nabla_X \xi$, satisfying $\Phi^2 = -Id + \eta \otimes \xi$ and $g(\Phi X, \Phi Y) = g(X, Y) - \eta(X) \eta(Y)$. Note that $(\eta, \xi)$ is contact.
			\item \textbf{Sub-Riemannian:} $(M, \mathcal{D}, g_{\mathcal{D}})$ where $\mathcal{D} \subset TM$ is a distribution (subbundle) and $g_{\mathcal{D}}$ is a metric on $\mathcal{D}$. Usually assume $\mathcal{D}$ is bracket-generating.
			\item \textbf{Symplectic:} $M^{2n}$ with closed, non-degenerate $\omega \in \Omega^2$.
			\item \textbf{Walker Manifold:} Pseudo-Riemannian $(M,g)$ admitting a parallel, null distribution $\mathcal{D} \subset TM$. ($\nabla_X Y \in \Gamma(\mathcal{D})$ for $Y \in \Gamma(\mathcal{D})$, and $g(Y_1, Y_2)=0$ for $Y_1, Y_2 \in \Gamma(\mathcal{D})$).
		\end{enumerate}

		\begin{center}
			{\bf Acknowledgments}
		\end{center}
		The authors are thankful to Prof. Aissa Wade for helpful comments and advice.

	\end{document}